\documentclass[a4paper,8pt]{article}
\usepackage{amsfonts}
\usepackage{amssymb,amsthm,color,comment,graphicx,tikz}
\usepackage{amsmath}
\usepackage[normalem]{ulem}
\usepackage{esint}
\usepackage{cancel}

\newcommand{\ds}{\displaystyle}

\newcommand{\R}{\mathbb{R}}

\newcommand{\ra}{\rightarrow}

\newtheorem{theorem}{Theorem}
\newtheorem{lemma}[theorem]{Lemma}
\newtheorem{proposition}[theorem]{Proposition}
\newtheorem{corollary}[theorem]{Corollary}

\def\I{{\mathcal I}}

\newcommand{\Om}{\Omega}
\newcommand{\lb}{\lambda}

\newcommand{\sm}{\setminus}

\newcommand{\ov}{\overline}
\newcommand{\vphi}{\varphi}
\newcommand{\vps}{\varepsilon}

\newcommand{\bp}{\begin{proof}}
\newcommand{\ep}{\end{proof}}

\begin{document}
\title{On the torsion function for simply connected, open sets in $\R^2$.}

\author{{M. van den Berg\footnote{corresponding author}} \\
School of Mathematics, University of Bristol\\
Fry Building, Woodland Road\\
Bristol BS8 1UG\\
United Kingdom\\
\texttt{mamvdb@bristol.ac.uk}\\
\\
{D. Bucur}\\
Laboratoire de Math\'ematiques, Universit\'e Savoie Mont Blanc \\
UMR CNRS  5127\\
Campus Scientifique,
73376 Le-Bourget-Du-Lac\\
France\\
\texttt{dorin.bucur@univ-savoie.fr}}
\date{7 November 2024}\maketitle
\vskip 1.5truecm \indent
\begin{abstract}\noindent
For an open set $\Om \subset \R^2$  let $\lambda(\Om)$ denote  the bottom of the spectrum of the Dirichlet Laplacian acting in $L^2(\Om)$.  Let $w_\Om$ be the torsion function for $\Om$, and let $\|.\|_p$ denote the $L^p$ norm. It is shown that there exist {$\eta_1>0,\eta_2>0$} such that { (i) $\|w_{\Om}\|_{\infty} \lambda(\Om)\ge 1+\eta_1$ for any non-empty, open, simply connected set $\Om\subset \R^2$ with $\lb(\Om) >0$, (ii)
$\|w_{\Om}\|_1\lambda(\Om)\le {(1-\eta_2)}|\Om|$ for any non-empty, open, simply connected set $\Om\subset\R^2$ with finite measure $|\Om|$.}
\end{abstract}
\begin{center} \textbf{Keywords}: Torsion function, first Dirichlet eigenfunction, simply connected.
\end{center}
\mbox{}
\section{Introduction\label{sec1}}
Let $\Omega$ be a non-empty, open set in Euclidean space $\R^m$, and let $\lambda(\Om)$ denote the bottom of the spectrum of the Dirichlet Laplacian acting in $L^2(\Om)$. That is
$$\lb(\Om)= \inf\Big\{\frac{\int_\Om |\nabla u|^2 }{\int_\Om u^2} : u \in H^1_0(\Om), u \not \equiv 0\Big\}.$$
If the measure $|\Om|$ of $\Om$ is finite, then $\lb(\Om)$ is the fundamental eigenvalue.

We define the torsion function for arbitrary non-empty open sets $\Om$ (with either finite or infinite measure) by a limiting procedure. Let  $B_r$ denote the ball of radius $r$ centred at the origin, and
define $w_\Om:\Om \ra [0, +\infty]$  for a.e. $x \in \R^m$ by $w_\Om(x) = \lim_ {r \ra +\infty}w_{\Om \cap B_r}(x)$, where $w_{\Om \cap B_r}(x)$ is the unique solution of
\begin{equation*}
-\Delta w=1,\,\,w\in H_0^1(\Om\cap B_r).
\end{equation*}
If $\Om$ has finite measure, then $w_\Om$ is the classical torsion function, and the convergence $w_{\Om \cap B_r} \to w_\Om$ is strong in $H^1$. If  $\Om$ has infinite measure, {then} $w_\Om$ is a well-defined, non-negative Borel function, and possibly infinite valued.
If $w_{\Om}$ is finitely valued on $\Om$, then $w_{\Om}$ is a weak solution, in the sense of distributions, of  $-\Delta w=1$ on $\Om$, which satisfies $w|_{\partial\Om}=0$ at all regular points of $\partial \Om$.  Moreover, if $G_{\Om}(x,y),\, x\in \Om,\,y\in \Om$ is the kernel of the resolvent of the Dirichlet Laplacian, then
\begin{equation}\label{e1b}
w_{\Om}(x)=\int_{\Om}dy\,G_{\Om}(x,y).
\end{equation}

It is known that $\lb(\Om) >0$ is equivalent to $w_\Om \in L^\infty(\Om)$ (see  \cite[Theorem 1]{vdBC}). Moreover, as soon as $\lb(\Om)>0$, the following inequality holds
\begin{equation}\label{e2}
1{  <} \|w_{\Om}\|_{\infty}\lambda(\Om)\le c_m
\end{equation}
 and the constant $c_m$ satisfies, see \cite[Theorem 1]{vdBC},
\begin{equation*}
c_m\le 4+3m\log 2.
\end{equation*}

A proof of the lower bound in \eqref{e2} can be found in \cite[Theorem 5.3]{vdB}. It was shown in \cite[Theorem 1]{MvdB} and \cite[Theorem 3.3]{HLP} that if $m\ge 2$ the lower bound $1$ in \eqref{e2} is sharp:
for every $\varepsilon>0$ there exists a non-empty open set $\Om_{\varepsilon}\subset \R^m$ such that
\begin{equation}\label{e4}
\|w_{\Om_\vps}\|_{\infty}\lambda(\Om_\vps)<1+\varepsilon.
\end{equation}
An example of such an  open set $\Om_{\varepsilon}$ is an open cube with side-length $1$ punctured with $N^m$ balls positioned in a periodic order and with {suitably} chosen small radii. See Figure 1.
\begin{figure}[!ht]

\centering
\begin{tikzpicture}

\draw[black, thin] (0.25,0.25) -- (4.75,0.25);
\draw[black, thin] (4.75,0.25) -- (4.75,4.75);
\draw[black, thin] (4.75,4.75)-- (0.25,4.75);
\draw[black, thin] (0.25,4.75) -- (0.25,0.25);
\filldraw[black] (0.5,0.5) circle (1pt);
\filldraw[black] (0.5,1.0) circle (1pt);
\filldraw[black] (0.5,1.5) circle (1pt);
\filldraw[black] (0.5,2.0) circle (1pt);
\filldraw[black] (0.5,2.5) circle (1pt);
\filldraw[black] (0.5,3.0) circle (1pt);
\filldraw[black] (0.5,3.5) circle (1pt);
\filldraw[black] (0.5,4.0) circle (1pt);
\filldraw[black] (0.5,4.5) circle (1pt);
\filldraw[black] (0.5,0.5) circle (1pt);

\filldraw[black] (1.0,0.5) circle (1pt);
\filldraw[black] (1.0,1.0) circle (1pt);
\filldraw[black] (1.0,1.5) circle (1pt);
\filldraw[black] (1.0,2.0) circle (1pt);
\filldraw[black] (1.0,2.5) circle (1pt);
\filldraw[black] (1.0,3.0) circle (1pt);
\filldraw[black] (1.0,3.5) circle (1pt);
\filldraw[black] (1.0,4.0) circle (1pt);
\filldraw[black] (1.0,4.5) circle (1pt);

\filldraw[black] (1.5,0.5) circle (1pt);
\filldraw[black] (1.5,1.0) circle (1pt);
\filldraw[black] (1.5,1.5) circle (1pt);
\filldraw[black] (1.5,2.0) circle (1pt);
\filldraw[black] (1.5,2.5) circle (1pt);
\filldraw[black] (1.5,3.0) circle (1pt);
\filldraw[black] (1.5,3.5) circle (1pt);
\filldraw[black] (1.5,4.0) circle (1pt);
\filldraw[black] (1.5,4.5) circle (1pt);

\filldraw[black] (1.5,0.5) circle (1pt);
\filldraw[black] (1.5,1.0) circle (1pt);
\filldraw[black] (1.5,1.5) circle (1pt);
\filldraw[black] (1.5,2.0) circle (1pt);
\filldraw[black] (1.5,2.5) circle (1pt);
\filldraw[black] (1.5,3.0) circle (1pt);
\filldraw[black] (1.5,3.5) circle (1pt);
\filldraw[black] (1.5,4.0) circle (1pt);
\filldraw[black] (1.5,4.5) circle (1pt);

\filldraw[black] (2.0,0.5) circle (1pt);
\filldraw[black] (2.0,1.0) circle (1pt);
\filldraw[black] (2.0,1.5) circle (1pt);
\filldraw[black] (2.0,2.0) circle (1pt);
\filldraw[black] (2.0,2.5) circle (1pt);
\filldraw[black] (2.0,3.0) circle (1pt);
\filldraw[black] (2.0,3.5) circle (1pt);
\filldraw[black] (2.0,4.0) circle (1pt);
\filldraw[black] (2.0,4.5) circle (1pt);

\filldraw[black] (2.5,0.5) circle (1pt);
\filldraw[black] (2.5,1.0) circle (1pt);
\filldraw[black] (2.5,1.5) circle (1pt);
\filldraw[black] (2.5,2.0) circle (1pt);
\filldraw[black] (2.5,2.5) circle (1pt);
\filldraw[black] (2.5,3.0) circle (1pt);
\filldraw[black] (2.5,3.5) circle (1pt);
\filldraw[black] (2.5,4.0) circle (1pt);
\filldraw[black] (2.5,4.5) circle (1pt);

\filldraw[black] (3.0,0.5) circle (1pt);
\filldraw[black] (3.0,1.0) circle (1pt);
\filldraw[black] (3.0,1.5) circle (1pt);
\filldraw[black] (3.0,2.0) circle (1pt);
\filldraw[black] (3.0,2.5) circle (1pt);
\filldraw[black] (3.0,3.0) circle (1pt);
\filldraw[black] (3.0,3.5) circle (1pt);
\filldraw[black] (3.0,4.0) circle (1pt);
\filldraw[black] (3.0,4.5) circle (1pt);

\filldraw[black] (3.5,0.5) circle (1pt);
\filldraw[black] (3.5,1.0) circle (1pt);
\filldraw[black] (3.5,1.5) circle (1pt);
\filldraw[black] (3.5,2.0) circle (1pt);
\filldraw[black] (3.5,2.5) circle (1pt);
\filldraw[black] (3.5,3.0) circle (1pt);
\filldraw[black] (3.5,3.5) circle (1pt);
\filldraw[black] (3.5,4.0) circle (1pt);
\filldraw[black] (3.5,4.5) circle (1pt);

\filldraw[black] (4.0,0.5) circle (1pt);
\filldraw[black] (4.0,1.0) circle (1pt);
\filldraw[black] (4.0,1.5) circle (1pt);
\filldraw[black] (4.0,2.0) circle (1pt);
\filldraw[black] (4.0,2.5) circle (1pt);
\filldraw[black] (4.0,3.0) circle (1pt);
\filldraw[black] (4.0,3.5) circle (1pt);
\filldraw[black] (4.0,4.0) circle (1pt);
\filldraw[black] (4.0,4.5) circle (1pt);

\filldraw[black] (4.5,0.5) circle (1pt);
\filldraw[black] (4.5,1.0) circle (1pt);
\filldraw[black] (4.5,1.5) circle (1pt);
\filldraw[black] (4.5,2.0) circle (1pt);
\filldraw[black] (4.5,2.5) circle (1pt);
\filldraw[black] (4.5,3.0) circle (1pt);
\filldraw[black] (4.5,3.5) circle (1pt);
\filldraw[black] (4.5,4.0) circle (1pt);
\filldraw[black] (4.5,4.5) circle (1pt);
\end{tikzpicture}
\caption{$\Omega_{\varepsilon}$,\, $N=9,\,m=2$. }
\label{fig1}
\end{figure}

In 1994,  R. Ba\~nuelos and T. Carroll \cite[(4.4),(4.6)]{BC} conjectured that the lower bound  in the left-hand side of \eqref{e2} remains sharp for $\Om$ in the class of simply connected sets in $\R^2$.
In Theorem \ref{the1} below we show that if $\Om$ is simply connected in $\R^2$ then the constant $1$  in the left hand side of \eqref{e2} can be improved by a positive number, thereby disproving the Ba\~nuelos and Carroll conjecture. We shall give an estimate of this number, without any claim  to be optimal.

Our first main result is the following.
\begin{theorem} \label{the1} Let $ \Om \subset \R^2$ be non-empty, open and simply connected. If $\lb(\Om) >0$, then
\begin{equation}\label{e5}
\|w_{\Om}\|_{\infty}\lambda(\Om)>1+\frac{3^5(67-44\sqrt 2)}{2^{34}\cdot5\cdot7j_0}{\mathfrak{c}_1,}
\end{equation}
where $j_0$ is the first positive zero of the Bessel function $J_0$,
and
\begin{equation*}
{\mathfrak{c}_1=\max_{c\ge\frac92\log 2}\frac{(1-2^{9/2}e^{-c})^7}{c^9}.}
\end{equation*}
\end{theorem}
{The right-hand side of \eqref{e5} is strictly larger than $1+10^{-16}$.}

In \cite[Theorem 1]{BC} it was shown that if $\Om$ is simply connected in $\R^2$
\begin{equation}\label{e5x}
\|w_{\Om}\|_{\infty}\le \frac{7\zeta(3)}{16c_0^2}r(\Om)^2,
\end{equation}
where $\zeta(.)$ is the Riemann zeta function, $c_0$ is the schlicht Bloch-Landau constant, and
\begin{equation*}
r(\Om)=\sup_{x\in\Om}d_{\Om}(x)
\end{equation*}
is the inradius of $\Om$ with
the distance to the boundary $\partial \Omega$ given by
\begin{equation*}
d_{\Om}(x)=\inf\{|x-y|:y\notin\Om\}.
\end{equation*}

Combining Theorem \ref{the1} with \eqref{e5x} yields
\begin{equation*}
\lambda(\Om)\ge \frac{16c_0^2}{7\zeta(3)}\Big(1+\frac{3^5(67-44\sqrt 2)}{2^{34}\cdot5\cdot7j_0}{\mathfrak{c}_1}\Big)r(\Om)^{-2}.
\end{equation*}
This improves the constant in \cite[(0.7)]{BC}, see also \cite[p.2473]{B}, by a factor given by the right-hand side of \eqref{e5}. We believe this is, as far as we know, the best result at present.
However, its importance is not so much its small improvement rather the non-trivial contribution of the right-hand side of \eqref{e5}.

We recall \cite[Theorem 1]{BC2} where it was shown that
\begin{equation*}
\sup\big\{\|w_{\Om}\|_{\infty}r(\Om)^{-2}: \Om\,\, \textup{simply connected, non-empty, open in}\, \R^2,\,r(\Om)<\infty\}
\end{equation*}
has a maximiser. We point out that this maximiser is not necessarily a minimiser of
\begin{equation*}
\inf\big\{\lambda(\Om)r(\Om)^2: \Om\,\, \textup{simply connected, non-empty, open in}\, \R^2,\,r(\Om)<\infty\big\}.
\end{equation*}

We shall see below that it suffices to prove Theorem \ref{the1} for non-empty, open, simply connected sets in $\R^2$ with finite measure.
 The proof of Theorem \ref{the1} relies on some properties of the torsion function $w_{\Om}$ and of the first  Dirichlet eigenfunction $u_{\Om}$. These will be presented in Proposition \ref{the3} below. The assertions there will be given in more generality as we believe these to be of independent interest.
In Proposition \ref{the3} (i) we merely assume that $\Om$ is simply connected and $\lambda(\Om)>0$, while in (ii) we assume that the torsion function has a maximum {point}. In fact (ii) is not needed in the proof of Theorem \ref{the1}.
In Proposition \ref{the3}(iii) we obtain a lower bound for the distance to the boundary of the maximum of the first Dirichlet eigenfunction (chosen positive) in a simply connected open set in $\R^2$. The existence of that first Dirichlet eigenfunction is guaranteed if the Dirichlet Laplacian acting in $L^2(\Om)$ has compact resolvent. The latter is, by \cite[Corollary 8]{MvdB4}, equivalent to
\begin{equation}\label{e13b}
|\{x\in\Om:d_{\Om}(x)>\varepsilon\}|<\infty, \forall \varepsilon>0.
\end{equation}

 \begin{proposition}\label{the3} Let $\Om\subset \R^2$ be non-empty, open and simply connected.
\begin{itemize}
\item[\textup{(i)}] If $\lambda(\Om)>0$, then for all $c>0$,
\begin{equation}\label{e12}
w_{\Om}(x)\le \frac{32}{3}d_{\Om}(x)^{1/2}\lambda(\Om)^{-3/4}c^{3/2}+\lambda(\Om)^{-1}2^{9/2}e^{-c},\,x\in\Om.
\end{equation}
\item[\textup{(ii)}]
If $\lambda(\Om)>0$ and if $x_{w_{\Om}}$ is a point at which $w_{\Om}$ attains its maximum, then for any $c\ge \frac92\log 2$,
\begin{equation}\label{e12a}
d_{\Om}(x_{w_{\Om}})\ge \frac{3^2}{2^{10}}\frac{(1-2^{9/2}e^{-c})^2}{c^3}  \lambda(\Om)^{-1/2}.
\end{equation}
Moreover,
\begin{equation}\label{e12b}
d_{\Om}(x_{w_{\Om}})\ge \frac{3^2}{2^{10}}\mathfrak{c}_2 \lambda(\Om)^{-1/2},
\end{equation}
where
\begin{equation}\label{e12c}
\mathfrak{c}_2=\max_{c\ge\frac92\log 2}\frac{(1-2^{9/2}e^{-c})^2}{c^3}
\end{equation}
\item[\textup{(iii)}]
If \eqref{e13b} holds, then there exists a point $x_{u_{\Om}}\in \Om$ such that $u_{\Om}(x_{u_{\Om}})=\|u_{\Om}\|_{\infty}.$
Moreover,
\begin{equation}\label{e13}
d_{\Om}(x_{u_{\Om}})\ge \frac{3^2}{2^{10}}\mathfrak{c}_2 \lambda(\Om)^{-1/2}.
\end{equation}
\end{itemize}
\end{proposition}

Estimates like \eqref{e13} have been obtained in \cite[Corollary 1.2]{RS}. The constant there has not been quantified. However, \cite[Proposition 3.1]{RS} suggests a {method of estimating it}. In the special case of open, bounded and convex sets in $\R^m$ we refer to \cite[Theorem 2.8]{BMS} and  \cite[Theorem 3.2]{BL}.
\smallskip

In order to introduce the second result of the paper, we recall a classical inequality for open sets $\Om$ with finite measure involving the torsional rigidity $T(\Om):=\|w_{\Om}\|_1$ of $\Om$, which goes back to  P\'{o}lya and Szeg\"{o} \cite{PSZ}.
It asserts that the function $F$ defined by
\begin{equation}\label{e6}
F(\Omega)=\frac{T(\Omega)\lambda(\Omega)}{\vert\Omega\vert}
\end{equation}
satisfies
\begin{equation}\label{e7}
F(\Omega)< 1.
\end{equation}
The constant $1$ in the right-hand side of \eqref{e7} is sharp, as shown in \cite[Theorem 1.2]{vdB2}. That is, given $\varepsilon\in (0,1)$ there exists a non-empty open set $\Om'_{\varepsilon}\subset \R^m$ such that
\begin{equation}\label{e8}
F(\Omega'_{\varepsilon})> 1-\varepsilon.
\end{equation}
The open set $\Om'_{\varepsilon}$ is yet again an open cube with side-length $1$ punctured with $N^m$ balls positioned in a periodic order and with suitably chosen small radii, as in Figure 1.
 Moreover \cite[Theorem 1.1]{vdB2} quantifies the inequality in \eqref{e7}, and reads
 \begin{equation}\label{e9}
F(\Omega)\le  1-\frac{2m\omega_m^{2/m}}{m+2}\frac{T(\Om)}{|\Om|^{(m+2)/m}},
\end{equation}
  where $\omega_m$ is the measure of the ball with radius $1$ in $\R^m$.
The question arises once more whether the constant $1$ in \eqref{e7} is sharp for simply connected sets $\Om\subset \R^2$ with finite measure. Our second main result is the following.
\begin{theorem} \label{the2}   If $\Om\subset \R^2$ is non-empty, open and simply connected with finite measure, then
\begin{equation}\label{e10}
 F(\Om)<1-  \frac{ 3^4}{12801\cdot  2^{31} \Big( 1+ \frac{\pi +\pi^2}{16}\Big)^4 }.
\end{equation}
\end{theorem}

Note that for simply connected open sets \eqref{e9} cannot give a bound of the form \eqref{e10}. Indeed, it is possible to construct a simply connected open set  $\Om\subset \R^2$ with $|\Om|=1$
and $T(\Om)$ arbitrarily small: let $\Om_n$ be the union of $n$ disjoint balls with total measure $1-\frac{1}{100^n}$ connected by $n-1$ disjoint thin tubes having total measure $\frac{1}{100^n}$.
It is easily seen that  $T(\Om_n)/|\Om_n|^2=O(\frac{1}{n})$.

Though the sets $\Om_{\varepsilon}$ and $\Om'_{\varepsilon}$  look very similar, and the statements of Theorems \ref{the1} and \ref{the2} are very similar, we were unable to show that either
Theorem \ref{the1} implies Theorem \ref{the2} with a possibly different correction to $1$ in \eqref{e10} or Theorem \ref{the2} implies Theorem \ref{the1} with a possibly different correction to $1$ in \eqref{e5}.

For the $p$-{Laplacian},  the value $1$ in inequality \eqref{e7} is sharp for $p \in (1,m]$. This was proved in \cite{LBGBFP}.  For $p>m$, the constant $1$ is no longer sharp, see \cite{LBDB}.
 The latter situation corresponds to the case where points have strictly positive $p$-capacity. A somewhat  similar situation occurs in Theorem \ref{the2}. Here, the simply connectedness hypothesis
  implies a uniform thickness with respect to the Wiener criterion, leading to a constant strictly less than $1$. However, in our framework we are able to give an estimate for  that constant, which is not the case for arbitrary $p$.

The key result on which the proof of Theorem \ref{the2} relies is established in Proposition \ref{prop1} below. It is convenient to introduce the mean to max and the {$L^1-L^2$ participation ratios of the torsion function respectively by}
\begin{equation}\label{e10a}
\Phi_{1,\infty}(\Om)=\frac{1}{|\Om|}\frac{\|w_{\Om}\|_1}{\|w_{\Om}\|_{\infty}}, \hskip 1cm \Phi_{1,2}(\Om)=\frac{1}{|\Om|^\frac 12} \frac{\|w_{\Om}\|_1}{\|w_{\Om}\|_{2}}.
\end{equation}

\begin{proposition}\label{prop1}   If $\Omega\subset \R^2$ is a non-empty, open and simply connected set with finite measure, then
\begin{equation*}
  \Phi_{1,2}(\Om)\le \Bigg(1- \frac{ 3^4 }{12801\cdot  2^{31}  \Big( 1+ \frac{\pi +\pi^2}{16}\Big)^4 }\Bigg)^\frac 12.
\end{equation*}
\end{proposition}
Roughly speaking,  the proof follows the same strategy as in \cite{LBDB} for the $p$-Laplacian when $p>m$. However, there are some key points where the proofs are significantly different. These concern the growth of the torsion function near the boundary which, in the case $p>m$, is  reduced to the analysis of a singleton. In our case,  the proof requires quantified decay estimates relying on the Wiener criterion along with the repeated use of Vitali's covering theorem.

\smallskip

We also obtain an estimate of the mean to max ratio.
\begin{corollary}\label{cor1} If $\Om\subset \R^2$ is non-empty, open and simply connected  with finite measure, then
\begin{equation}\label{e10b.2}
 \Phi_{1,\infty}(\Om)<\Bigg(1-  \frac{ 3^4 }{12801   \cdot 2^{31}  \Big( 1+ \frac{\pi +\pi^2}{16}\Big)^4 }\Bigg)
 \Big(1+\frac{3^5(67-44\sqrt 2)}{2^{34}\cdot5\cdot7j_0}{\mathfrak{c}_1}\Big)^{-1}.
\end{equation}
\end{corollary}

\smallskip

Corollary  \ref{cor1}  complements the results of the mean to max ratio of the torsion function in \cite{HLP}, and of the $p$-torsion from \cite{LBDB}. In the latter, the authors showed that this ratio is also bounded away from $1$ in the case $p>m$.

 \smallskip
 In order to have a global view of the inequalities in \eqref{e2},
we introduce the following families of sets:
\begin{equation}\label{e3c}\mathfrak{A}_o={\{\Om: \Om\,\textup{is non-empty, open in}\, \R^m\}},\end{equation}
\begin{equation}\label{e3d}\mathfrak{A}_{sc}={\{\Om: \Om\,\textup{is simply connected, non-empty, open in}\, \R^2\}},\end{equation}
\begin{equation}\label{e3e}\mathfrak{A}_{co}={\{\Om: \Om\,\textup{is non-empty, open, convex in}\, \R^m\}}. \end{equation}

It remains an open question to prove existence of a minimiser of the variational problem
\begin{equation*}
\inf\big\{\|w_{\Om}\|_{\infty}\lambda(\Om): \Om\in \mathfrak{A}_{sc}\,, \lambda(\Om)>0\big\}.
\end{equation*}

{The variational problem}
$$\inf\big\{\|w_{\Om}\|_{\infty}\lambda(\Om): \Om\in \mathfrak{A}_{o}\,, \lambda(\Om)>0\big\}$$
does not have a minimiser, and {the infimum is asymptotically attained by the} sequences described in \eqref{e4}.

It has been shown in \cite[3.12]{P} that
$$\inf\big\{\|w_{\Om}\|_{\infty}\lambda(\Om): \Om\in \mathfrak{A}_{co}\,, \lambda(\Om)>0\big\}=\frac{\pi^2}{8},$$
with equality if $\Omega$ is the open connected set, bounded by two parallel $(m-1)$-dimensional hyperplanes.

 \smallskip
The constant $c_m$ received further  improvements in \cite[Theorem 1.5]{HV}, where it was shown  that
\begin{equation*}
c_m\le \frac{m}{8}+\frac14\sqrt{5\big(1+\frac14\log 2\big)}\sqrt{m}+1.
\end{equation*}
The sharp value of the constant $c_m$ is not known. {By analogy with} the minimum problem, one may investigate the existence of a maximiser of the variational problem
\begin{equation}\label{e3b}
\sup\big\{\|w_{\Om}\|_{\infty}\lambda(\Om): \Om\in \mathfrak{A}\,, \lambda(\Om)>0\big\},
\end{equation}
for the admissible classes  defined in \eqref{e3c}, \eqref{e3d}   and  \eqref{e3e}.
It is not known whether \eqref{e3b} admits a maximiser for \eqref{e3c} or for  \eqref{e3d}. In the class of convex sets  \eqref{e3e}, Henrot, Lucardesi and Philippin proved the existence of a maximiser in \cite[Theorem 3.2]{HLP}, and they conjectured that, for $m=2$, the maximiser is an equilateral triangle.

Similarly the variational problem
\begin{equation*}
\sup\big\{F(\Om): \Om\in \mathfrak{A}\,, |\Om|<\infty \big\},
\end{equation*}
remains open for \eqref{e3d}. For  \eqref{e3c} there exist maximising sequences described in \eqref{e8} but no maximiser,  see \cite{vdB2}.
For \eqref{e3e} and $m=2$ it has been conjectured that the supremum equals $\pi^2/12$ achieved for a sequence of thinning rectangles.
{See \cite{RB} for some recent progress on this conjecture.}

 \smallskip
This paper is organised as follows. The proofs of Theorem \ref{the1}  and Proposition \ref{the3} are deferred to Section \ref{sec2} while the proof of Theorem \ref{the2}, Proposition \ref{prop1} and Corollary \ref{cor1} are deferred to Section \ref{sec3}.

\section{ Proof of Theorem \ref{the1}.\label{sec2}}
Throughout the paper we denote an open ball centred at $x$ with radius $r$ by $B(x;r)$ or $B_r(x)$.   We put $B_r={B(0;r)}$.

We first prove Proposition \ref{the3}, and begin with a few basic lemmas. Let $p_{\Om}(x,y;t),\,x\in\Om,\,y\in\Om,\,t>0$ be the Dirichlet heat kernel for $\Omega$.
Then
\begin{equation*}
G_{\Om}(x,y)=\int_{(0,\infty)}dt\,p_{\Om}(x,y;t)
\end{equation*}
is the kernel of the resolvent of the Dirichlet Laplacian  acting in $L^2(\Om)$.
\begin{lemma}\label{lem1} Let $\Om\subset \R^2$ be non-empty, open and simply connected and with finite measure. If $0<\Lambda<\infty$, then
\begin{equation}\label{e16}
\int_{\{y\in\Om:|x-y|<\Lambda\}}dy\,G_{\Om}(x,y)\le \frac43d_{\Om}(x)^{1/2}\Lambda^{3/4},
\end{equation}
and
\begin{equation}\label{e17}
\int_{\Om}dy\,G_{\Om}(x,y)\le \frac{4}{3\pi^{3/4}}d_{\Om}(x)^{1/2}|\Om|^{3/4}.
\end{equation}
\end{lemma}
\begin{proof}
It was shown both in \cite[(5.14)]{MvdB3} and \cite[(2.8)]{MP} that
\begin{align}\label{e18}
G_{\Om}(x,y)&\le \frac{1}{2\pi}\log \frac{(d_{\Om}(x)+|x-y|)^{1/2}+d_{\Om}(x)^{1/2}}{(d_{\Om}(x)+|x-y|)^{1/2}-d_{\Om}(x)^{1/2}}\nonumber\\&
=\frac{1}{\pi}\log\Big(\Big(1+\frac{d_{\Om}(x)}{|x-y|}\Big)^{1/2}+\Big(\frac{d_{\Om}(x)}{|x-y|}\Big)^{1/2}\Big).
\end{align}
The proof of \eqref{e18} uses the conformal map of $\Om$ onto the unit disc, together with K\"obes one quarter theorem.
Using the inequality
\begin{equation}\label{e19}
\log((1+\theta)^{1/2}+\theta^{1/2})\le \theta^{1/2},\, \theta\ge 0,
\end{equation}
we obtain that
\begin{equation}\label{e20}
G_{\Om}(x,y)\le \frac{1}{\pi}\Big(\frac{d_{\Om}(x)}{|x-y|}\Big)^{1/2},
\end{equation}
Integrating both sides of \eqref{e20} yields
\begin{align}\label{e21}
\int_{\{y\in\Om:|x-y|<\Lambda\}}dy\,G_{\Om}(x,y)&\le \frac{1}{\pi}\int_{\{y\in\Om:|x-y|<\Lambda\}}dy\,\Big(\frac{d_{\Om}(x)}{|x-y|}\Big)^{1/2}\nonumber\\&
\le \frac{1}{\pi}\int_{B(x;\Lambda)}dy\,\Big(\frac{d_{\Om}(x)}{|x-y|}\Big)^{1/2}\nonumber\\&
=\frac43d_{\Om}(x)^{1/2}\Lambda^{3/2}.
\end{align}
To prove \eqref{e17} we rearrange the set $\Om$ radially around the point $x$, and then use \eqref{e21} with
\begin{equation*}
\Lambda=\Big(\frac{|\Om|}{\pi}\Big)^{1/2}.
\end{equation*}
\end{proof}
We note that, thanks to \eqref{e19}, the numerical constant in the right-hand side of \eqref{e17} improves upon the ones stated in \cite[(5.17]{MvdB3} and \cite[(2.9)]{MP} respectively.

Below we state and prove an upper bound for $G_{\Om}(x,y)$ which decays faster than $|x-y|^{-1/2}$ for large $|x-y|$.
\begin{lemma}\label{lem2}
If $\Om$ is an open set in $\R^2$ for which the spectrum of the Dirichlet Laplacian is discrete, then
\begin{equation*}
G_{\Om}(x,y)\le \frac{2^{1/2}}{4\pi}\int_0^{\infty}\frac{dt}{t}e^{-\frac{|x-y|^2}{8t}-\frac{t\lambda(\Om)}{4}}.
\end{equation*}
\end{lemma}
\begin{proof}
The following inequality is well known. By the semigroup property of the Dirichlet heat kernel and Cauchy-Schwarz,
\begin{align}\label{e26}
p_{\Om}(x,y;t)&=\int_{\Om}dz\,p_{\Om}(x,z;t/2)p_{\Om}(z,y;t/2)\nonumber\\&
\le \Big(\int_{\Om}dz\,p_{\Om}(x,z;t/2)^2\Big)^{1/2}\Big(\int_{\Om}dz\,p_{\Om}(z,y;t/2)^2\Big)^{1/2}\nonumber\\&
=\Big(p_{\Om}(x,x;t)p_{\Om}(y,y;t)\Big)^{1/2}.
\end{align}
Furthermore by the spectral representation of the Dirichlet heat kernel and monotonicity
\begin{align}\label{e27}
p_{\Om}(x,x;t)&\le e^{-t\lambda(\Om)/2}p_{\Om}(x,x;t/2)\nonumber\\&
\le e^{-t\lambda(\Om)/2}p_{\R^2}(x,x;t/2)\nonumber\\&
=\frac{e^{-t\lambda(\Om)/2}}{2\pi t}.
\end{align}
We obtain, by \eqref{e26} and \eqref{e27},
\begin{align*}
p_{\Om}(x,y;t)&\le \Big(p_{\Om}(x,x;t)p_{\Om}(y,y;t)\Big)^{1/4}p_{\Om}(x,y;t)^{1/2}\nonumber\\&
\le \Big(p_{\Om}(x,x;t)p_{\Om}(y,y;t)\Big)^{1/4}p_{\R^2}(x,y;t)^{1/2}\nonumber\\&
\le \frac{2^{1/2}}{4\pi t}e^{-\frac{t\lambda(\Om)}{4}-\frac{|x-y|^2}{8t}}.
\end{align*}
\end{proof}

To prove Proposition \ref{the3} we now let $\Lambda>0$ be arbitrary.
By Lemma \ref{lem2} and Tonelli's Theorem
\begin{align}\label{e30}
\int_{\{y\in\Om:|x-y|\ge\Lambda\}}dy\,G_{\Om}(x,y)&\le \frac{2^{1/2}}{4\pi}\int_0^{\infty}\frac{dt}{t}e^{-\frac{t\lambda(\Om)}{4}}\int_{\{y\in\Om:|x-y|\ge\Lambda\}} dy\,e^{-\frac{|x-y|^2}{8t}}\nonumber\\&
\le \frac{2^{1/2}}{4\pi}\int_0^{\infty}\frac{dt}{t}e^{-\frac{t\lambda(\Om)}{4}}\int_{\{y\in\R^2:|x-y|\ge\Lambda\}} dy\,e^{-\frac{|x-y|^2}{8t}}\nonumber\\&
=2^{3/2}\int_0^{\infty}dt\, e^{-\frac{t\lambda(\Om)}{4}-\frac{\Lambda^2}{8t}}\nonumber\\&
\le 2^{3/2}\int_0^{\infty} dt\,e^{-\frac{t\lambda(\Om)}{8}}\cdot \sup_{t>0}e^{-\frac{t\lambda(\Om)}{8}-\frac{\Lambda^2}{8t}}\nonumber\\&
=\frac{2^{9/2}}{\lambda(\Om)}e^{-\lambda(\Om)^{1/2}\Lambda/4}.
\end{align}

By \eqref{e1b}, \eqref{e16} and \eqref{e30}
\begin{align}\label{e32}
w_{\Om}(x)&=\int_{\Om}dy\,G_{\Om}(x,y)\nonumber\\&\le \frac{4}{3}d_{\Om}(x)^{1/2}\Lambda^{3/2}+\frac{2^{9/2}}{\lambda(\Om)}e^{-\lambda(\Om)^{1/2}\Lambda/4}.
\end{align}
For $c\ge 0$, we substitute
\begin{equation*}
\Lambda=4c\lambda(\Om)^{-1/2}
\end{equation*}
into \eqref{e32}. This gives \eqref{e12}.

If a maximum point $x_{w_{\Om}}$ exists, then by \eqref{e2}
\begin{equation*}
w_{\Om}(x_{w_{\Om}})=\|w_{\Om}\|_{\infty}\ge \lambda(\Om)^{-1}.
\end{equation*}
This, together with \eqref{e12}, gives \eqref{e12a}. Taking the maximum over all $c\ge \frac92\log 2$ in \eqref{e12a} gives \eqref{e12b}.

\smallskip

Following the lines of \cite[Theorem 5.3]{vdB}, we obtain by \eqref{e1b} that
\begin{align}\label{e35}
w_{\Om}(x)&
\ge \int_{\Om}dy\,G_{\Om}(x,y)\frac{u_{\Om}(y)}{\|u_{\Om}\|_{\infty}}\nonumber\\&
=\frac{1}{\lambda(\Om)}\frac{u_{\Om}(x)}{\|u_{\Om}\|_{\infty}}.
\end{align}
By \eqref{e12} and \eqref{e35} we obtain that
\begin{equation}\label{e13a}
\frac{32}{3}d_{\Om}(x)^{1/2}\lambda(\Om)^{1/4}c^{3/2}+2^{9/2}e^{-c}\ge \frac{u_{\Om}(x)}{\|u_{\Om}\|_{\infty}}.
\end{equation}
If $x_{u_{\Om}}$ is a maximum point for the first Dirichlet eigenfunction then substituting $x=x_{u_{\Om}}$ into \eqref{e13a} gives for $c\ge 0$ that
\begin{equation}\label{e13f}
d_{\Om}(x_{u_{\Om}})^{1/2}\ge \frac{3}{32}\frac{1-2^{9/2}e^{-c}}{c^{3/2}}\lambda(\Om)^{-1/4},\,c\ge 0.
\end{equation}
The right-hand side of \eqref{e13f} is non-negative for $c\ge\frac92\log 2$. So we obtain that
\begin{equation}\label{e13c}
d_{\Om}(x_{u_{\Om}})\ge \frac{3^2}{2^{10}}\frac{(1-2^{9/2}e^{-c})^2}{c^{3}}\lambda(\Om)^{-1/2},\,c\ge\frac92\log 2.
\end{equation}
Taking the maximum over all such $c$ gives \eqref{e13} with \eqref{e12c}.

Next we show that \eqref{e13b} implies existence of a maximum point of $u_{\Om}$. Let $(x_n)$ be a sequence of points in $\Om$ such that $\lim_n u_{\Om}(x_n)=\|u_{\Om}\|_{\infty}$. Without loss of generality we may assume that for all $n$, $u_{\Om}(x_n)\ge \frac12\|u_{\Om}\|_{\infty}$.
Inequality \eqref{e13a} implies that $d_{\Om}(x_n)\ge c'\lambda(\Om)^{-1/2}$ for some $c'>0$. Hence  all $(x_n)$ are in the closed set $\{x\in\Om: d_{\Om}(x)\ge c'\lambda(\Om)^{-1/2}\}$.  To prove compactness it suffices to show that this set is bounded. Suppose to the contrary. Then $\{x\in\Om: d_{\Om}(x)\ge c'\lambda(\Om)^{-1/2}\}$ contains a sequence $(y_n)$ such that $|y_n-y_m|\ge \lambda(\Om)^{-1/2}, m\ne n$. Hence
\begin{equation}\label{e13d}
\{x\in\Om: d_{\Om}(x)\ge 2^{-1}c'\lambda(\Om)^{-1/2}\}\supset\bigcup_nB(y_n;2^{-1}c'\lambda(\Om)^{-1/2}).
\end{equation}
The balls in the right-hand side of \eqref{e13d} are disjoint, and their union has infinite measure. This contradicts \eqref{e13b}. Hence $\{x\in\Om: d_{\Om}(x)\ge c\lambda(\Om)^{-1/2}\}$ is compact,   and $(x_n)$ contains a convergent subsequence also denoted by $(x_n)$ and converging to say $x^*$. Since the restriction of $u_{\Om}$ to $\{x\in\Om: d_{\Om}(x)\ge c\lambda(\Om)^{-1/2}\}$ is continuous, $x_n\rightarrow x^*$ implies $u_{\Om}(x_n)\rightarrow u_{\Om}(x^*)$. By hypothesis $\lim_n u_{\Om}(x_n)=\|u_{\Om}\|_{\infty}$. It follows that $u_{\Om}(x^*)=\|u_{\Om}\|_{\infty}$, and so $x^*=x_{u_{\Om}}$. Hence $u_{\Om}(x_{u_{\Om}})=\|u_{\Om}\|_{\infty}$.
\hfill$\square$

\medskip

\noindent{\it Proof of Theorem \ref{the1}.} Let $\Om\subset \R^2$ be an open, simply connected set such that $\lb(\Om)>0$, and $\|w_\Om\|_\infty < +\infty$.  Since $r \mapsto w_{\Om \cap B_r}$ is an increasing sequence of functions and $ w_{\Om \cap B_r}\to w_\Om$ pointwise almost everywhere, we have that
$$\lim_{r \to +\infty} \| w_{\Om \cap B_r}\|_\infty = \|w_\Om\|_\infty.$$
Moreover, since $C_c^\infty (\Om)$ is dense in $H^1_0(\Om)$, we have $\lim_{r \to +\infty}\lb( \Om \cap B_r)= \lb(\Om)$.
This implies that it is enough to prove Theorem \ref{the1} for bounded, simply connected sets. Indeed, if we do so, we choose from $\Om\cap B_r$  the (possibly $r$-dependent) component which supports the first eigenvalue of $\Om \cap B_r$ for which the inequality is true. By monotonicity of the torsion function with respect to inclusion, the inequality holds on the full set $\Om \cap B_r$. Passing to the limit $r \to +\infty$ we get it on $\Om$. From now on, we can assume without loss of generality that $\Om$ is bounded.

We refine \eqref{e35} so that
\begin{equation}\label{e37}
w_{\Om}(x)=\frac{1}{\lambda(\Om)}\frac{u_{\Om}(x)}{\|u_{\Om}\|_{\infty}}+\int_{\Om}dy\,G_{\Om}(x,y)\Big(1-\frac{u_{\Om}(y)}{\|u_{\Om}\|_{\infty}}\Big).
\end{equation}
At the point $q:=x_{u_{\Om}}$ we obtain, by \eqref{e35} and \eqref{e37},
\begin{align}\label{e38}
\|w_{\Om}\|_{\infty}&\ge w_{\Om}(x_{u_{\Om}})=\frac{1}{\lambda(\Om)}+\int_{\Om}dy\,G_{\Om}(q,y)\Big(1-\frac{u_{\Om}(y)}{\|u_{\Om}\|_{\infty}}\Big)\nonumber\\&
= \frac{1}{\lambda(\Om)}+\int_{\Om}dy\,G_{\Om}(q,y)\Big(1-\frac{u_{\Om}(y)}{\|u_{\Om}\|_{\infty}}\Big)_{+}\nonumber\\&
\ge \frac{1}{\lambda(\Om)}+\int_{\Om}dy\,G_{\Om}(q,y)\Big(1-\lambda(\Om)w_{\Om}(y)\Big)_{+},
\end{align}
where $(.)_{+}$ denotes the positive part.
The Green's function for the disc satisfies
\begin{equation}\label{e40}
G_{B(q;d_{\Om}(q))}(q,y)=\frac{1}{2\pi}\log\Big(\frac{d_{\Om}(q)}{d_{\Om}(q)-d_{\Om}(y)}\Big)\ge \frac{1}{2\pi}\frac{d_{\Om}(y)}{d_{\Om}(q)}.
\end{equation}
Since the Dirichlet heat kernel, and hence the Green function, is monotone under set inclusion we find by \eqref{e38} and \eqref{e40} that
\begin{align}\label{e39}
\int_{\Om}dy\,G_{\Om}(q,y)\Big(1-\lambda(\Om)w_{\Om}(y)\Big)_{+}\ge \frac{1}{2\pi}\int_{B(q;d_{\Om}(q))}\frac{d_{\Om}(y)}{d_{\Om}(q)}\Big(1-\lambda(\Om)w_{\Om}(y)\Big)_{+}.
\end{align}
By Proposition \ref{the3}(i) we have for $c>0$,
\begin{align}\label{e40g}
\Big(1-\lambda(\Om)w_{\Om}(y)\Big)_{+}&\ge \Big(1-2^{9/2}e^{-c}-\frac{32}{3}d_{\Om}(y)^{1/2}\lambda(\Om)^{1/4}c^{3/2}\Big)_+\nonumber\\&
=(1-2^{9/2}e^{-c})\Big(1-\Big(\frac{d_{\Om}(y)}{R_c}\Big)^{1/2}\Big)_+,
\end{align}
where, suppressing $\lambda$-dependence,
\begin{equation}\label{e40a}
R_c=\frac{3^2}{2^{10}}\frac{(1-2^{9/2}e^{-c})^2}{c^3}\lambda(\Om)^{-1/2}.
\end{equation}
By \eqref{e39} and \eqref{e40g},
\begin{align}\label{e40h}
\int_{\Om}dy\,G_{\Om}(q,y)&\Big(1-\lambda(\Om)w_{\Om}(y)\Big)_{+}\nonumber\\&
\ge \frac{1-2^{9/2}e^{-c}}{2\pi}\int_{B(q;d_{\Om}(q))}\frac{d_{\Om}(y)}{d_{\Om}(q)}\Big(1-\Big(\frac{d_{\Om}(y)}{R_c}\Big)^{1/2}\Big)_+.
\end{align}
Let $\overline{q}\in{ \R^2}\setminus\Om$ be such that $|\overline{q}-q|=d_{\Om}(q)$. By \eqref{e13c} and \eqref{e40a}, $|\overline{q}-q| \ge R_c$.  Let $q^*$ be the point on the straight line segment from $\overline{q}$ to $q$ such that $|\overline{q}-q^*|=\frac12R_c$.
We claim that
\begin{equation}\label{e40b}
B(q^*;\frac12R_c)\subset B(q;d_{\Om}(q)).
\end{equation}
Since $|q^*-q|=d_{\Om}(q)-|\overline{q}-q^*|=d_{\Om}(q)-\frac12R_c,$ we have for $y\in B(q^*;\frac12R_c)$ that $|y-q|< |y-q^*|+|q^*-q|< \frac12R_c+d_{\Om}(q)-\frac12R_c=d_{\Om}(q)$, which is \eqref{e40b}.
By \eqref{e40b},
\begin{align}\label{e49}
\int_{B(q;d_{\Om}(q))}&dy\,\frac{d_{\Om}(y)}{d_{\Om}(q)}\Big(1-\Big(\frac{d_{\Om}(y)}{R_c}\Big)^{1/2}\Big)_{+}\nonumber\\&\ge \int_{B(q^*;\frac12R_c)}dy\,\frac{d_{\Om}(y)}{d_{\Om}(q)}\Big(1-\Big(\frac{d_{\Om}(y)}{R_c}\Big)^{1/2}\Big).
\end{align}
The triangle inequality gives that
\begin{equation}\label{e40e}
\frac12R_c-|y-q^*|\le d_{\Om}(y)\le \frac12R_c+|y-q^*|.
\end{equation}
By \eqref{e40e},
\begin{align}\label{e49a}
\int_{B(q^*;\frac12R_c)}&dy\,\frac{d_{\Om}(y)}{d_{\Om}(q)}\Big(1-\Big(\frac{d_{\Om}(y)}{R_c}\Big)^{1/2}\Big)\nonumber\\&
\ge \int_{B(q^*;\frac12R_c)}dy\,\frac{\frac12R_c-|q^*-y|}{d_{\Om}(q)}\Big(1-\Big(\frac{\frac12R_c+|q^*-y|}{R_c}\Big)^{1/2}\Big)\nonumber\\&
=2\pi\int^{\frac12R_c}_0dr\,r\frac{\frac12R_c-r}{d_{\Om}(q)}\Big(1-\Big(\frac{\frac12R_c+r}{R_c}\Big)^{1/2}\Big).
\end{align}
Changing the variable $r=\frac12R_c(\theta^2-1)$ gives that
\begin{align}\label{e49b}
2\pi\int^{\frac12R_c}_0dr\,&r\frac{\frac12R_c-r}{d_{\Om}(q)}\Big(1-\Big(\frac{\frac12R_c+r}{R_c}\Big)^{1/2}\Big)\nonumber\\&
=\frac{2\pi\Big(\frac12R_c\Big)^3}{d_{\Om}(q)}\Big(\frac16-2^{1/2}\int_1^{\sqrt 2}d\theta\,\theta^2(\theta^2-1)(2-\theta^2)\Big)\nonumber\\&
=\frac{3^5(67-44\sqrt 2)\pi}{2^{33}\cdot5\cdot7}\frac{(1-2^{9/2}e^{-c})^6}{c^9d_{\Om}(q)}\lambda(\Om)^{-3/2}.
\end{align}

\begin{figure}[!ht]
\centering
\begin{tikzpicture}
\draw[black, thin] (0,0) circle(3.0);
\node at (-2,0.9) {$B(q;d_{\Om}(q))$};
\node at (0,-3.3) {$\overline{q}$};
\filldraw[black] (0,0) circle (1pt) node[anchor=west]{$q$};
\filldraw[black] (0,-3.0) circle (1pt);
\draw[black, thin] (0,-1.89) circle(1.11);
\draw[gray,very thin](0,0)--(0,-3);
\filldraw[black] (0,-1.89) circle(1pt);
\node at (0.3,-1.89) {$q^*$};
\end{tikzpicture}
\caption{$B(q^*;\frac12R_c)\subset B(q;d_{\Om}(q))\subset\Omega,\,d_{\Om}(q)=|q-\overline{q}|\ge R_c$.}
\label{fig2}
\end{figure}

Since $\Omega$ contains a ball with radius $r(\Om)$ we have by monotonicity of Dirichlet eigenvalues that
\begin{equation}\label{e50}
\lambda(\Om)\le \frac{j_0^2}{r(\Om)^2},
\end{equation}
where $j_0^2$ is the first Dirichlet eigenvalue of a ball with radius $1$.
By \eqref{e40h}, \eqref{e49}, \eqref{e49a}, \eqref{e49b} and \eqref{e50} we obtain that
\begin{equation}\label{e49c}
\int_{\Om}dy\,G_{\Om}(q,y)\Big(1-\lambda(\Om)w_{\Om}(y)\Big)_{+}\ge \frac{3^5(67-44\sqrt 2)}{2^{34}\cdot5\cdot7j_0}\frac{(1-2^{9/2}e^{-c})^7}{c^9}\lambda(\Om)^{-1}.
\end{equation}
Theorem \ref{the1} follows from \eqref{e38} and \eqref{e49c} by taking the maximum over
\newline  $c\in [\frac92\log 2,\infty)$.
\hspace*{\fill }$\square $

\section{Proof of Theorem \ref{the2}.\label{sec3}}
The key result to prove Theorem \ref{the2} is Proposition \ref{prop1}.
Let $\Om\subset \R^2$ be a  non-empty, open and simply connected set with finite measure. Assume we know that
\begin{equation*}
\ds \frac{\Big (\ds \int_\Om w_\Om \Big)^2}{|\Om| \ds \int_\Om w_\Om^2 }\le 1-\eta.
\end{equation*}

Taking $w_\Om$ as a test function for the first eigenvalue we have
$$T(\Om) \lb(\Om) \le T(\Om) \frac{\ds \int_\Om |\nabla w_\Om|^2 }{\ds \int_\Om w_\Om^2 }= \frac{\Big(\ds \int_\Om  w_\Om \Big)^2}{\ds \int_\Om w_\Om^2 }\le (1-\eta)|\Om|,$$
so that  Theorem \ref{the2} holds true.

Before proving Proposition \ref{prop1}, we give a technical result.
\begin{lemma}\label{vdb202}
 Let $U\subset \R^2$ be a non-empty, open, simply connected set, let $x_0\in \partial U$ and  let $R_0 >0$. Assume that $v\in H^1(U \cap B_{2R_0}(x_0))$ satisfies the following
\begin{equation}\label{vdb203}
\begin{cases}
-\Delta v \le  1&\hbox{in }U \cap B_{2R_0}(x_0),\\
v=0&\hbox{on }\partial U, \\
v \le  c  &\hbox{on }\partial  B_{2R_0}(x_0),
\end{cases}
\end{equation}
where $c \ge 1$.
Then $\forall x \in B_{R_0}(x_0)\cap U$ we have  that
\begin{equation}\label{vdb204}
v(x) \le  c  \Big( 1+ \frac{\pi +\pi^2}{2R_0^2}\Big) \frac 43 \pi^{-\frac 34} (4\pi R_0^2)^\frac34 |x-x_0|^\frac 12.
\end{equation}
\end{lemma}

\begin{proof}
It is enough to prove the inequality for $c=1$. Indeed, if $c >1$, then $\frac vc$ satisfies \eqref{vdb203} with $c=1$. Assume now that $c=1$. 

Note that $U \cap B_{2R_0}(x_0)$ may not be connected. However its connected components are simply connected sets. Let $U_1$ be such a connected component, and let $x \in U_1$. If $x_0$ does not belong to $\partial U_1$, then it is enough to prove  inequality \eqref{vdb204}, replacing the point $x_0$ by a point of the intersection of the straight line
segment $[x,x_0]$ with $\partial U_1$. Collecting all inequalities on  all components, \eqref{vdb204}  will hold on the full set $B_{R_0}(x_0)\cap U$ with a point $x_0 \in \partial U$.

 In the sequel, we assume that  $U \cap B_{2R_0}(x_0)$ is connected.  Let 
$\varphi : B_{2R_0}(x_0) \to [0,1]$ be the radial function such that
$\varphi(x)=0$ on $B_{R_0}(x_0)$, and
$$\varphi(x)= \frac 12-\frac 12 \cos\Big (\pi \frac{|x|-R_0}{R_0} \Big),\, x\in B_{2R_0}(x_0)\sm \ov B_{R_0}(x_0).$$
Note that $\varphi \in H^2(B_{2R_0}(x_0))$ and, by direct computation, that
$$\forall x \in  B_{2R_0}(x_0),\;\;  |\Delta \vphi (x)|\le \frac{\pi+\pi^2}{2R_0^2}.$$
Now, $v -\varphi \le 0$ on $\partial (U \cap B_{2R_0}(x_0))$ and
$$-\Delta (v-\vphi) \le 1+ \frac{\pi+\pi^2}{2R_0^2} \mbox {in } {\mathcal D}'(U\cap B_{2R_0}(x_0)).$$
This implies that
$$\forall x \in U\cap B_{2R_0}(x_0), \; v(x)-\vphi(x) \le \Big (1+ \frac{\pi+\pi^2}{2R_0^2} \Big) w_{U\cap B_{2R_0}(x_0)}.$$
But $U\cap B_{2R_0}(x_0)$ is a union of simply connected sets, and     \eqref{e17} gives that
$$ w_{U\cap B_{2R_0}(x_0)}(x)\le \frac 43  \pi^{-\frac 34} |B_{2R_0}(x_0)|^\frac 34 \Big (d_{\partial (U\cap B_{2R_0}(x_0))}(x)\Big)^\frac 12.$$

Since $x_0 \in \partial U$ and $\varphi (x)=0$ on $B_{R_0}(x_0)$, we get \eqref{vdb204}.
\end{proof}
We now come back to the proof of Proposition \ref{prop1}, and follow the main lines of the proof in \cite[Lemma 3.3]{LBDB} adapted to our situation and keeping track of the constants.
\begin{proof}
We assume without loss of generality that $\Om$ is also bounded and rely on the approximation $w_{\Om \cap B_R}  \to w_\Om$ strongly in $H^1$.
By rescaling, we can assume that $\ds \fint_\Om dx\, w_\Om  =1$, where $\ds \fint_\Om dx\,  w_\Om $ denotes the volume average of $w_{\Om}$ over $\Om$. Our goal is  to prove that
$$\fint_\Om dx\,w_\Om^2  \ge \frac {1}{1-\eta}.$$
We rely on the following equality
$$\fint_\Om dx\,w_\Om^2   = \Big (\fint_\Om dx\,w_\Om  \Big)^2 + \fint_\Om dx\,\Big (w_\Om-\fint_\Om dx\,w_\Om  \Big)^2= 1+  \fint_\Om dx\Big (w_\Om-1 \Big)^2.$$
Our goal is to prove that
$$\int_\Om dx\, \Big (w_\Om-1 \Big)^2 \ge \frac{\eta}{1-\eta} |\Om|.$$
We cover $\Om$ by balls as follows: $\forall x\in \Om$ we consider $B_{R(x)}(x)$ the ball centred at $x$ with radius $R(x) = d(x, \partial \Om)$. The union of all these balls coincides with $\Om$.  Using Vitali's covering theorem  (1.5.1 in \cite{EG}), there exists an at most countable family of pairwise disjoint such balls $\{B_{R_i}(x_i)\}_{i \in I}$ such that
$$\Om \subset \cup_{i \in I} B_{5R_i}(x_i).$$
Let us introduce $R_0= \sqrt{8}$, chosen such that $w_{B_{R_0}}(0)=2$. We  consider the following  cases.

\smallskip
\noindent {\bf Case 1.} Assume that
$$\sum _{i\in \I, R_i \ge R_0} |B_{R_i}(x_i)| \ge   \frac {1}{2} \cdot \frac {1}{ 5^2} |\Om|.$$
We claim that if this situation occurs then
\begin{equation}\label{vdb205}
\Big|\Big\{ w_\Om \ge \frac 32\Big\}\Big| \ge \frac 14\cdot  \frac {1}{2} \cdot \frac {1}{ 5^2} |\Om|.
\end{equation}
Indeed, since $R_i^2\ge R_0^2=8$, we have that
$$\Big|\Big\{ w_{B_{R_i}(x_i)} \ge \frac 32\Big\}\Big| = \pi (R_i^2-6)\ge \frac{\pi}{4} R_i^2= \frac 14 |B_{R_i}(x_i)|.$$
Since $w_\Om \ge w_{B_{R_i}(x_i)}$, we get \eqref{vdb205}. Then
$$\int_\Om dx\, \Big (w_\Om-1 \Big)^2 \ge \int_{\{ w_\Om \ge \frac 32\}} \frac 14 dx \ge \frac 14 \cdot \frac 14 \cdot   \frac {1}{2} \cdot \frac {1}{ 5^2}|\Om|,$$
leading to
$$\eta_1 = \frac{1}{801}.$$

\smallskip
\noindent {\bf Case 2.} Assume now that
$$\sum _{i\in \I, R_i \le R_0} |B_{R_i}(x_i)| \ge    \frac {1}{2} \cdot \frac {1}{ 5^2} |\Om|.$$
By construction, every such ball $B_{R_i}(x_i)$ touches the boundary of $\Om$. Choose one point $y_i \in \partial \Om \cap \ov B_{R_i}(x_i)$. We  consider the following  two sub-cases.

\smallskip
\noindent {\bf Case 2-a.} Assume that
$$\sum_{i\in \I} \{|B_{R_i}(x_i)| : R_i \le R_0, w_\Om \le 2 \mbox{ on } B_{2R_0}(y_i) \}  \ge  \frac {1}{2} \cdot   \frac {1}{2} \cdot \frac {1}{ 5^2} |\Om|.$$
Consequently, we can use the estimate of Lemma \ref{vdb202}  with $c=2$ 
$$\forall x \in B_{R_0}(y_i) \cap \Om,\;\;  w_\Om(x) \le 2 \Big( 1+ \frac{\pi +\pi^2}{2R_0^2}\Big)  \frac 43\pi^{-\frac 34} (4\pi R_0^2)^\frac34 |x-y_i|^\frac 12.$$
Let   $r_0$ be  such that
$$ 2 \Big( 1+ \frac{\pi +\pi^2}{2R_0^2}\Big) \frac 43\pi^{-\frac 34} (4\pi R_0^2)^\frac34 r_0^\frac 12= \frac 12.$$
Then, for $x \in B_{R_i}(x_i) \cap B_{r_0} (y_i)$ we have $w_\Om(x) \le \frac 12$. Since $R_i \le R_0$, we get
$$|B_{R_i}(x_i) \cap B_{r_0} (y_i)|\ge \frac{r_0^2}{32} |B_{R_i}(x_i) |.$$
Finally,
$$ \Big|\Big\{0<w_\Om \le \frac 12\Big\}\Big| \ge  \frac{r_0^2}{32} \cdot   \frac {1}{2} \cdot   \frac {1}{2} \cdot \frac {1}{ 5^2}|\Om|,$$
gives that
$$\eta_2: = \frac{ r_0^2}{12801} < \frac{ r_0^2}{12800+r_0^2 },$$
since $r_0<1$.

\smallskip
\noindent {\bf Case 2-b.} Assume  now that
$$ \sum_{i\in \I} \{|B_{R_i}(x_i)| : R_i \le R_0, \max\{w_\Om(x) : x\in  B_{2R_0}(y_i)\} \ge 2 \}  \ge  \frac {1}{2} \cdot   \frac {1}{2} \cdot  \frac {1}{5^2} |\Om|.$$
If $w_\Om$ is extended by $0$ on $\R^2 \sm \Om$, the function satisfies in $\R^2$, in the sense of distributions,
$$-\Delta w_\Om \le 1 \mbox{ in } {\mathcal D}' (\R^2).$$
In particular, this implies that for every point $a \in \R^2$,
$$x \mapsto w_\Om(x) + \frac{|x-a|^2}{ 4}$$
is subharmonic in $\R^2$.  Assuming that $w_\Om(a)\ge 2$, then for every $r >0$
$$2 \le \fint_{B_r(a)}dx\, \Big(w_\Om (x)+ \frac{|x-a|^2}{ 4}\Big)  \le \fint _{B_r(a)} dx\,w_\Om (x)  + \frac{r^2}{4}.$$
Choosing  $r= \sqrt{2}$, we get
$$ \fint _{B_{ \sqrt{2}}(a)} dx\, w_\Om (x)  \ge \frac 32,$$
or
$$ \fint _{B_{ \sqrt{2}}(a)} dx\, (w_\Om (x) -1) \ge \frac 12,$$
and by Cauchy-Schwarz,
$$ \int _{B_{ \sqrt{2}}(a)}dx\, (w_\Om (x) -1)^2 \ge\frac 14 |B_{\sqrt{2}}(a)|.$$
Let us denote
$$A= \cup \big \lbrace B_{R_i}(x_i) : {i\in \I, R_i \le R_0, \max\{w_\Om(x) : x\in  B_{2R_0}(y_i)\} \ge 2} \}\big \rbrace.$$
We cover the set $A$ with the balls $B_{4R_0}(y_i)$ for the index $i$ above. By Vitali's covering  theorem, there exists a family of pairwise disjoint such balls indexed by $j \in J$ such that
$$A \subset \cup_{j \in J} B_{20R_0}(y_j).$$
Note that $J$ is necessarily finite. For every such index $j$ let us denote a maximum point of $w_\Om$ in $B_{2R_0}(y_j)$ by $x_j^*$. We know that $w_\Om (x_j^*) \ge 2$. Then
$$\int_{B_{ \sqrt{2}}(x_j^*) } dx\, (w_\Om-1)^2  \ge  \frac 12 \pi $$
so
$$\int_\Om dx\,(w_\Om-1)^2  \ge card(J) \frac 12 \pi.$$
 We also have that
$$ card(J) \pi (20 R_0)^2 \ge |A| \ge \frac 14 \cdot \frac {1}{5^2} |\Om|,$$
hence
$$\int_\Om dx\,(w_\Om-1)^2  \ge \frac 12 \cdot \frac{1} {(20 R_0)^2 }  \cdot \frac 14  \cdot \frac {1}{5^2} |\Om|.$$
We get
$$\eta_3: = \frac{\frac 12 \cdot \frac{1} {(20 R_0)^2 }  \cdot \frac 14  \cdot \frac {1}{5^2}}{1+ \frac 12 \cdot \frac{1} {(20 R_0)^2 }  \cdot \frac 14  \cdot \frac {1}{5^2}}.$$
  Then $\eta=\min\{\eta_1, \eta_2, \eta_3\}$ gives that
$$\eta= \frac{ 3^4 }{12801  \cdot  2^{31} \Big( 1+ \frac{\pi +\pi^2}{2R_0^2}\Big)^4 }.$$
\end{proof}

We finish this section with the proof of Corollary \ref{cor1}.

\begin{proof}
This  is a direct consequence of Theorems \ref{the1} and \ref{the2}. By the definition of $T(\Om)$, \eqref{e6} and \eqref{e10a}, we see that
\begin{equation}\label{e10b.3}
\Phi_{1,\infty}(\Om)=\frac{F(\Om)}{\|w_{\Om}\|_{\infty}\lambda(\Om)}.
\end{equation}
{We obtain \eqref{e10b.2} by applying the upper bound in Theorem \ref{the2} to the numerator and the lower bound in Theorem \ref{the1} to the denominator in \eqref{e10b.3} respectively}.
\end{proof}

\smallskip

\noindent {\bf Acknowledgements.} The authors wish to thank Dr. Katie Gittins for helpful suggestions.

\smallskip

\noindent {\bf Author Contributions} Not applicable.

\smallskip

\noindent {\bf Funding Information} MvdB and DB were supported by The London Mathematical Society under a Scheme 4 Research in Pairs Grant Reference 42214.

\smallskip

\noindent {\bf Data Availability} Data sharing is not applicable to this article.

\smallskip

\noindent {\bf Declarations}

\smallskip

\noindent {\bf Ethical Approval} Not applicable.

\smallskip

\noindent {\bf Competing Interests} I declare that the authors have no competing interests as defined by Springer, or other
interests that might be perceived to influence the results and/or discussion reported in this paper.

\end{document}